\RequirePackage{fix-cm} 
\documentclass[a4paper,twoside,12pt,reqno]{amsart}

\usepackage{fixltx2e}     

\usepackage[english]{babel}
\usepackage[latin1]{inputenc}

\usepackage{indentfirst,verbatim}

\usepackage{amsmath,amsfonts,amssymb,amsgen,amsbsy,eucal,mathrsfs,dsfont}
\usepackage{stmaryrd}

\usepackage[square,numbers]{natbib}

\usepackage{a4wide,verbatim}

\usepackage{epsfig,rotating,color}

\usepackage{multicol}
\usepackage{tikz,pgfplots}
\usepackage{graphicx}
\usepackage{mathtools}

\usepackage{textcomp}

\usepackage{caption}

\usepackage{hyperref}

\usepackage[square,numbers]{natbib}

\usepackage{systeme}

\hypersetup{
colorlinks=true, 
breaklinks=true, 
urlcolor= blue, 
linkcolor= blue, 
bookmarksopen=true, 
pdftitle={}, 
pdfauthor={}, 
pdfsubject={}, 
citecolor=red,  
}

\newcommand{\R}{\mathbb R}

\newcommand{\supp}{\operatorname{supp}}
\newcommand{\dist}{\operatorname{dist}}

\newcommand{\dd}{\,\mathrm d}

\usepackage{amsthm}
\newtheorem{thm}{Theorem}[section]

\newtheorem{lem}[thm]{Lemma}

\newtheorem{pro}[thm]{Proposition}

\theoremstyle{remark}
\newtheorem*{rem}{Remark}

\usepackage{nicefrac}

\usepackage{enumitem}
\newlist{steps}{enumerate}{1}
\setlist[steps, 1]{label = Step \arabic*:}

\newcommand*{\bydef}{\overset{\rm def}{=}}
\newcommand*{\norm}[1]{\left\Vert #1\right\Vert}

\newcommand{\eqdef}{\stackrel{\text{\tiny{def}}}{=}}

\title[Hunter--Saxton equation: Strong Ill-posedness]{\bf Strong Ill-Posedness in critical and subcritical regimes for the Hunter--Saxton equation}

\author[Guelmame and Houamed]{Billel Guelmame and Haroune Houamed}

\newcommand{\nfont}{\fontshape{n}\selectfont}

\address{({\nfont\textbf{Billel Guelmame}})  New York University Abu Dhabi, Abu Dhabi, United Arab Emirates} 
\email{billel.guelmame@nyu.edu}

\address{({\nfont\textbf{Haroune Houamed}}) Ko\c c University, Istanbul, Turkey} 
\email{hhouamed@ku.edu.tr, haroune.houamed@nyu.edu}

\let\oldtocsection=\tocsection
 
\let\oldtocsubsection=\tocsubsection

\renewcommand{\tocsection}[2]{\hspace{0em}\oldtocsection {#1}{#2}}
\renewcommand{\tocsubsection}[2]{\hspace{2em}\oldtocsubsection{#1}{#2}}

\numberwithin{equation}{section}

\usepackage{tabularx}
\usepackage[pagewise]{lineno} 

\begin{document}

\maketitle

\begin{abstract}

We study the  Hunter--Saxton equation on the real line and its  global dissipative solution in the energy space
$
L^\infty(\mathbb R)\cap \dot H^1(\mathbb R).
$
We prove strong ill-posedness through instantaneous failure of Sobolev regularity at and below the Lipschitz threshold. More precisely, for every $s\in(1,\nicefrac32]$, we construct
$
u_0\in L^\infty(\mathbb R)\cap\dot H^1(\mathbb R)\cap\dot H^s(\mathbb R)
$
whose unique global dissipative solution satisfies $u\notin C([0,T]; \dot H^s (\mathbb R))$, for every $T>0$.

The constructions differ substantially in the subcritical and critical regimes. For $1<s<\nicefrac32$, we superpose rescaled, localized bubbles with increasingly negative slopes; subcritical scaling preserves $\dot H^s$-summability, while the explicit characteristic formula produces norm inflation. At $s=\nicefrac32$, where scaling yields no smallness, we use logarithmically distributed compactly supported multiscale profiles whose breaking times converge to zero. A one-sided localization principle and an almost-orthogonality estimate then transfer the inflation of individual profiles to the full dissipative solution.

\end{abstract}

\medskip

 {\small{\bf AMS Classification :} 35Q35; 35L65; 35L67.

\medskip

{\bf Key words :} Hunter--Saxton equation; dissipative solution; wave breaking; critical Sobolev space;
strong ill-posedness; norm inflation.}

\tableofcontents

\section{Introduction}

The Hunter--Saxton equation was introduced by Hunter and Saxton as an asymptotic model for the propagation of orientation waves in a nematic liquid crystal \cite{HunterSaxton1991}. In differential form it reads
\begin{equation}\label{HS-differential-intro}
  \partial_x(\partial_t u+u \partial_x u  ) =\frac12 (\partial_x u)^2,
  \qquad (t,x)\in (0,\infty)\times\mathbb R.
\end{equation}
Since integration in the space variable determines the equation only up to a function of time, a normalization must be fixed on the real line. 
Throughout this paper we use the one-sided normalization
\begin{equation}\label{HS}
  \partial_t u + u  \partial_x u = \frac12 \int_{-\infty}^{x}\big(\partial_x u(t,y)\big)^2 \, \mathrm dy,
  \qquad u(0,x)=u_0(x).
\end{equation}
Thus the otherwise arbitrary function of time is set equal to zero.  
Other common formulations, including the symmetric normalization, differ by a time-dependent spatially constant correction.  
Fixing the normalization is essential when characteristic formulas and one-sided support properties are
used.

Besides its origin in liquid-crystal theory, the Hunter--Saxton equation has a rich analytic and geometric structure. 
It is completely integrable and admits a bi-Hamiltonian description \cite{HunterZheng1994}. 
It also arises as the geodesic equation for a right-invariant homogeneous $\dot H^1$ metric on a homogeneous space of the diffeomorphism group of the circle \cite{KhesinMisiolek2003,Lenells2007Sphere,Lenells2007Geodesic,Lenells2008}. 
These structures yield explicit Lagrangian representations. 
They also capture the mechanism of wave breaking, where the solution remains bounded and continuous while its slope diverges to $-\infty$.

Indeed, if $u$ is smooth and $q=\partial_x u$, differentiating
\eqref{HS} gives
\begin{equation}\label{q-intro}
  \partial_t q+u \partial_x q=-\frac12 q^2.
\end{equation}
Let $X$ be the characteristic flow, $\partial_t X(t,\xi)=u(t,X(t,\xi))$ and $X(0,\xi)=\xi$. 
As long as the classical solution exists, it holds that
\begin{equation}\label{riccati-intro}
  q(t,X(t,\xi))
  =\frac{2q_0(\xi)}{2+tq_0(\xi)}.
\end{equation}
Consequently, a nonnegative initial slope produces a global classical
evolution, whereas any negative initial slope leads to wave breaking at (see Proposition \ref{prop:blowup time} below)
\begin{equation}\label{breaking-time-intro}
  T_*=-\frac{2}{\inf_{x\in\mathbb R}q_0(x)}
  \quad\text{when }\inf_{x\in\mathbb R}q_0(x)<0.
\end{equation}
The early analysis of Hunter and Zheng established global weak solutions and clarified the role of vanishing viscosity and dispersion \cite{HunterZheng1995I,HunterZheng1995II}. 
After wave breaking, weak solutions are not selected uniquely without an additional admissibility principle. 
Two canonical continuations are the conservative solution, which retains the concentrated energy, and the dissipative solution, which removes the energy at breaking. 
Global dissipative semigroups and stability metrics were developed in \cite{BressanConstantin2005,BressanHoldenRaynaud2010}; uniqueness of the relevant generalized characteristics for dissipative solutions was proved by Dafermos \cite{Dafermos2011}. 
The conservative theory has a parallel Lagrangian formulation, and a recent uniqueness theorem is given in \cite{GrunertHolden2022}. 
In the present work, ``solution'' always means the unique global dissipative solution associated with the normalization \eqref{HS}.

The natural energy class for this dissipative flow is
\begin{equation*}
  E \bydef \bigl\{u\in L^\infty(\mathbb R):\partial _x u \in L^2(\mathbb R)\bigr\}  = L^\infty(\mathbb R)\cap \dot H^1(\mathbb R).
\end{equation*}
For smoother data, the equation is locally well posed by standard transport estimates. 
In Sobolev notation the threshold suggested by \eqref{q-intro} is $u_0\in\dot H^s$ with $s>\nicefrac{3}{2}$: then $q_0\in H^{s-1}\hookrightarrow L^\infty (\mathbb R)$, so the velocity is Lipschitz and the transport equation for $q$ can be closed. 
The endpoint $s=\nicefrac{3}{2}$ is qualitatively different because of failure of Sobolev embeddings, that is
\begin{equation*}
	  \dot H^{\nicefrac{3}{2}}(\mathbb R)\not \hookrightarrow \dot W^{1,\infty}(\mathbb R),
  \qquad  \text{equivalently} \qquad
  H^{\nicefrac{1}{2}}(\mathbb R) \not \hookrightarrow L^\infty(\mathbb R).
\end{equation*}
The critical Besov refinement $B^{\nicefrac{3}{2}}_{2,1}$ does control the Lipschitz norm, whereas $B^{\nicefrac{3}{2}}_{2,r}$ with $r>1$ does not; see \cite{BahouriCheminDanchin2011} for the underlying endpoint embeddings and paradifferential estimates.

Several forms of low-regularity instability are already known. On the circle, local well-posedness and nonuniform dependence were studied by Holliman \cite{Holliman2010}, and Holmes and Tiglay proved nonuniform continuity of the data-to-solution map in periodic Besov spaces \cite{HolmesTiglay2018}. 
On the real line, Ye and Yin  \cite{YeYin2022} introduced mixed spaces adapted to the fact that the one-sided primitive in \eqref{HS} does not need belong to a classical $L^p$-based inhomogeneous space, and obtained local well-posedness, a blow-up criterion and ill-posedness (in the sense of norm inflation) in  critical spaces.    
 More recently, Guo, Ye and Yin \cite{GuoYeYin2024} proved a sharp ill-posedness result in mixed Besov spaces, driven by a low-frequency obstruction and the absence of the energy condition $\partial_x u \in L^2$.  
 
The question considered here is different: our target framework remains within the natural energy class, and we are interested in whether the global dissipative semigroup preserves the critical $\dot H^{\nicefrac{3}{2}}$ and subcritical $\dot H^{s}$, $s\in (1,\nicefrac{3}{2}$),  Sobolev
regularities.
Our main result gives a negative answer in a strong form:

\begin{thm}[Strong failure of regularity persistence]\label{thm:main}
For any $s\in (1, \nicefrac {3}{2}]$, there exists an initial datum
\begin{equation*}
  u_0\in L^\infty(\R) \cap \dot H^1(\R) \cap \dot H^{s}(\R)
\end{equation*}
for which the corresponding unique global dissipative solution $u$ of \eqref{HS} satisfies
\begin{equation}\label{main}
  \sup_{0<t<T}\|u(t,\cdot)\|_{\dot H^{s}}=\infty,
  \qquad\text{for every }T>0.
\end{equation}
In particular, $u\notin C([0,T];\dot H^{s}(\R))$ for every $T>0$.
\end{thm}
 
\begin{rem}[Well-posedness vs Ill-posedness]
Setting $X^s \bydef L^\infty(\mathbb R)\cap \dot H^1(\mathbb R) \cap \dot H^{s}(\mathbb R)$,   the Hunter--Saxton equation is locally well-posed in $X^s$ for $s> \nicefrac {3}{2}$ due to classical transport estimates (see Proposition \ref{prop:propagation}). 
Thus, given its  well known  global well-posedness in the energy space $E=X^1$  from \cite{BressanConstantin2005,Dafermos2011}, Theorem \ref{thm:main} completes the picture of (strong) well-posedness vs ill-posedness of the Hunter--Saxton equation by establishing the non-existence of solutions in $C([0,T];X^{s})$, for $s\in (1, \nicefrac {3}{2}]$.
\end{rem}

 \begin{rem}[Novelty compared to recent work]

To the best of our knowledge, Theorem \ref{thm:main} provides the first ill-posedness  result for the Hunter--Saxton equation in a subcritical class of solutions satisfying the natural $\dot H^1$ energy bound. As for the critical regime (corresponding to $s=\nicefrac{3}{2}$), it  provides an improvement of the ill-posedness statements to some recent results. More precisely, first, a norm inflation phenomenon was established in \cite{YeYin2022} for initial data in  the space $L^\infty(\mathbb R)\cap \dot H^1(\mathbb R) \cap \dot B^{1+\nicefrac{1}{p}}_{p,q} (\mathbb R)$, for $p\in [1,\infty)$ and $q\in (1,\infty]$. Our Theorem \ref{thm:main} further promotes this statement to a strong ill-posedness (non-existence) result in case $p=q=2$. Due to the trivial embedding of Besov spaces, this immediately implies the same non-existence result for data in $L^\infty(\mathbb R)\cap \dot H^1(\mathbb R) \cap \dot B^{1+\nicefrac{1}{p}}_{p,q} (\mathbb R)$, for $p\in [1,2]$ and $q\in (1,2]$. Note, however, that the same lines of our proof below could be adapted to recover the remaining range of parameters $p\in (2,\infty)$ and $q\in (2,\infty]$. For the sake of simplicity, we choose  not to pursue this generalization here. Second, another comparable recent finding was established in 
 \cite{GuoYeYin2024}, where the authors proved a non-existence result for  \eqref{HS} in a critical class of solutions, namely $u_0\in  L^\infty(\mathbb R) \cap \dot B^{1+\nicefrac{1}{p}}_{p,1} (\mathbb R) \setminus \dot H^1(\mathbb R)$ with $p\in [1,\infty]$. In this case, they show that  no such an initial datum would generate a strong solution. There, the missing energy space $\dot H^1$ is the key point for the ill-posedness, which forces the right-hand side of \eqref{HS} to be unbounded. 
 
 The failure established by Theorem \ref{thm:main} is not merely a loss of uniform continuity of the flow map; the selected global solution instantaneously leaves the critical Sobolev class. 
We stress that this formulation is specific to the dissipative continuation; it does not assert   non-existence among all possible weak solutions. Our constructions of the initial data, and thus the proof of Theorem \ref{thm:main}, is  different compared to prior works,  and will be discussed next.
 	
 \end{rem}


\subsection*{Strategy of proof}
The general strategy is to construct a single initial datum as an infinite superposition of carefully designed, widely separated bubbles whose breaking times accumulate at zero. This approach is in the spirit of bubble-superposition arguments used in ill-posedness constructions for the Euler equations \cite{BL15}. In both regimes, the underlying dynamical mechanism is the same: large negative initial slopes are amplified by the Riccati evolution along dissipative characteristics. However, the construction of the bubbles and the analysis of their evolution differ substantially between the subcritical range $s\in(1,\nicefrac32)$ and the critical case $s=\nicefrac32$. We now briefly describe their common structure and the decisive differences between the two regimes.

We will rely on the  explicit representation of the unique global dissipative solution of \eqref{HS}, which was introduced by Bressan amd Constantin in \cite{BressanConstantin2005} (see also Dafermos \cite{Dafermos2011}).  
The brief statement to recall here is that, for $t>0$, introducing the active set 
\begin{equation}\label{It}
  I_t \bydef \{\xi\in\R:u_0'(\xi)\text{ exists and } u_0'(\xi)>-2/t\},  \qquad I_0 \bydef \R,
\end{equation}
the unique dissipative characteristic is defined, for every position $\xi\in \mathbb R$, by
\begin{equation}\label{charX}
	X(t,\xi) \bydef \xi+t u_0(\xi)   + \frac12 \int_0^t(t-\tau) \int_{I_\tau\cap(-\infty,\xi)}u_0'(\zeta)^2 \dd \zeta \dd \tau,
\end{equation}
while the global unique dissipative solution of \eqref{HS} is represented along these characteristics by the formula 
\begin{equation}\label{charU}
	u(t,X(t,\xi)) = u_0(\xi)  + \frac12\int_0^t \int_{I_\tau\cap(-\infty,\xi)}u_0'(\zeta)^2 \dd \zeta \dd \tau.
\end{equation} 

The proof of Theorem \ref{thm:main} takes substantially different forms in the subcritical range $s\in(1,\nicefrac{3}{2})$ and at the critical endpoint $s=\nicefrac{3}{2}$. 
Indeed, although both arguments exploit a similar dynamical mechanism, based on the observation that large negative initial slopes are amplified by the Riccati evolution along characteristics, the corresponding constructions of the initial data are fundamentally different.

In the subcritical range $s\in(1,\nicefrac{3}{2})$, the initial datum is constructed as a superposition of single-scale localized bubbles of the form
\begin{equation*}
	u_{0,n}(x)
	=
	A_n\ell_n
	W\left(\frac{x-y_n}{\ell_n}\right),
	\qquad x\in\mathbb R,
\end{equation*}
where $W\in C_c^\infty(\mathbb R)$ is chosen so that its derivative $Q=W'$ has a negative plateau:
\begin{equation*}
	Q=-1
	\qquad\text{on a nontrivial interval}.
\end{equation*}
In this regime, the scaling relation
\begin{equation*}
	\norm{u_{0,n}}_{\dot H^s(\mathbb R)}
	\sim
	A_n\ell_n^{\frac32-s}
\end{equation*}
is the key point. Since $\frac32-s>0$, one may choose $A_n\gg1$, thereby producing a large negative slope, while keeping the $\dot H^s$-norm small by taking $\ell_n\ll1$. 
The exact characteristic formula for the dissipative solution then propagates the initial negative plateau and produces the required norm inflation. 
In fact, the subcritical argument yields the stronger conclusion that the instability already occurs in the smaller space
\begin{equation*}
	\dot W^{1,p_s}(\mathbb R),
	\qquad
	p_s\bydef \frac{2}{3-2s},
\end{equation*}
in the sense of Sobolev embedding
\begin{equation*}
	\dot H^s(\mathbb R)
	\hookrightarrow
	\dot W^{1,p_s}(\mathbb R).
\end{equation*}

At the critical regime corresponding to $s=\nicefrac{3}{2}$, the preceding dilation argument no longer provides any smallness, since
\begin{equation*}
	\| A_n\ell W (\cdot/\ell)\|_{\dot H^{\frac{3}{2}}}
	\sim A_n
\end{equation*}
is invariant under the spatial scale $\ell$. 
The critical construction must therefore exploit a different feature, that also needs to be consistent with failure of the endpoint embedding
\begin{equation*}
	\dot H^{\frac{3}{2}}(\mathbb R)
\hookrightarrow
\dot W^{1,\infty}(\mathbb R).
\end{equation*}
More precisely, we use a logarithmically weighted multiscale Fourier profile, subsequently localized in physical space, to construct compactly supported smooth functions that are arbitrarily small in
\begin{equation*}
	L^\infty(\mathbb R)\cap\dot H^1(\mathbb R)\cap\dot H^{3/2}(\mathbb R)
\end{equation*}
but have increasingly large negative slopes. 
By the Riccati formula for $\partial_xu$, these profiles generate breaking times converging to zero. 
The precise construction is given in Lemmas \ref{lem:frequency-profile} and \ref{lem:bubbles}.

The critical profiles are then translated to rapidly separated locations and assembled into a single datum. 
To analyze the corresponding dissipative solution, we consider the solutions generated by successive partial sums of the initial bubbles and introduce the increments between two consecutive partial solutions. 
These increments are not themselves the solutions generated by the individual bubbles. 
Nevertheless, the one-sided domain of dependence principle shows that their spatial derivatives remain supported in uniformly bounded and mutually separated intervals. 
It also provides a pointwise telescoping representation of the full dissipative solution in terms of these increments. 
These localization properties follow from the explicit dissipative characteristic representation and are established in Section \ref{sec:localization}.

Finally, we observe that when a partial solution reaches its breaking time, the preceding partial solution is still smooth. 
A logarithmic continuation argument shows then that wave breaking forces the $\dot H^{1/2}$-seminorm of the corresponding differentiated increment to become unbounded. 
Since the increments are exponentially separated, an almost-orthogonality estimate in $\dot H^{1/2}(\mathbb R)$ transfers this inflation to the complete solution, and, subsequently, we conclude that the $\dot H^{3/2}$-norm of the dissipative solution is unbounded on every positive time interval.

\subsection*{Organization of the paper}

 Section \ref{sec:prelim} records the persistence of
supercritical regularity and the explicit wave-breaking criterion.  
Then, in Section \ref{section:subcritical},  we prove Theorem \ref{thm:main} in the subcritical regime corresponding to the case $s\in (1,\nicefrac{3}{2})$. 
To that end, we first establish the core lemma about the time-evolution of an initial negative plateau, then  we construct the precise initial data  and, eventually, discuss the way it generates instability of all subcritical norms due to the time-evolution of a superposition of a rescaled negative plateau.

Finally, Section \ref{section:critical} concerns the critical regime corresponding to the case $s= \nicefrac{3}{2}$. 
There, we construct the endpoint profiles and prove the separated-support lemma used to sum their critical seminorms. 
We also establish the required one-sided localization for dissipative solutions and complete the superposition argument, leading at the end to the proof of the Theorem \ref{thm:main} in the case $s=\nicefrac{3}{2}$.

\section{Preliminaries and wave breaking}\label{sec:prelim}

In this section, we discuss  two essential ingredients supporting the sharpness of the our main results. 
The first one (Proposition \ref{prop:propagation}) is about the local-in-time persistence of super-critical Sobolev regularities by dissipative solutions of \eqref{HS}, while the second one (Proposition \ref{prop:blowup time}) provides a precise expression of the finite time of blow-up of the Lipschitz norm of that local solution, for specific initial data. The results of this section are not new, we outline their details of proof, however, for the sake of completeness and convenience.

We use homogeneous Sobolev seminorms defined through the Fourier transform.
For $0<s<1$, we also use the equivalent Gagliardo seminorm
\begin{equation*}
  \|f\|_{\dot H^s(\mathbb R)}^2
  \simeq \iint_{\R^2}
  \frac{|f(x)-f(y)|^2}{|x-y|^{1+2s}}\dd x\dd y.
\end{equation*}

We will also need to homogeneous Sobolev spaces $\dot W^{1,p}(\mathbb R)$, for $p\in [2,\infty]$, which are endowed by their seminorms 
\begin{equation*}
	\norm {f}_{\dot W^{1,p}(\mathbb R)} \bydef \norm {f'}_{L^p(\mathbb R)}.
\end{equation*}

We recall that, for smooth solutions, multiplying \eqref{q-intro} by $2q$ gives
\begin{equation*}
  \partial_t(q^2)+\partial_x(uq^2)=0.
\end{equation*}
Thus the $\dot H^1$ energy is conserved before the blow-up and is non-increasing
for the global dissipative solutions.  Moreover, due  to the transport structure above, we have the following local persistence of higher initial Sobolev regularities. Note that a similar statement can be established in the case of initial data in $L^\infty (\mathbb R) \cap \dot H^1(\mathbb R) \cap \dot B^{1+\nicefrac{1}{p}}_{p,1}(\mathbb R)$, for any $p\in [1,\infty)$. See \cite[Theorem 1.1]{GuoYeYin2024}.

\begin{pro}[Supercritical persistence of regularity]\label{prop:propagation}
Let $s>\nicefrac{3}{2}$ and $u_0$ be a given initial datum in $ L^\infty \cap \dot H^1 \cap \dot H^s(\R)$.  
Then, there exists $T>0$ such that the dissipative solution $u$ of \eqref{HS} is the unique strong solution on $[0,T]$ and
\begin{equation}\label{propagation} 
  \sup_{0 \leqslant t \leqslant T}\|u(t,\cdot)\|_{\dot H^s}
  \leqslant 2\|u_0\|_{\dot H^s}.
\end{equation}
Furthermore, whenever
\begin{equation*}
\int_0^T\| \partial_x u(t,\cdot)\|_{L^\infty} \dd t < \infty,	
\end{equation*}
this solution can be continued beyond $T$.
\end{pro}

\begin{proof} We only outline the proof for $s<2$, for persistence of higher regularities can be established subsequently in a similar way. 
Set $r=s-1\in (\nicefrac{1}{2}, 1)$ and $q=\partial_x u$.  
A standard Friedrich's approximation of \eqref{q-intro}, followed by the classical one-dimensional transport estimate (see \cite[Theorem 3.14]{BahouriCheminDanchin2011}, for instance), yields the a priori control 
\begin{equation}\label{q-estimate}
  \frac{\dd\ }{\dd t} \|q(t,\cdot)\|_{H^r}  \leqslant C \|q(t,\cdot)\|_{L^\infty} \|q(t,\cdot)\|_{H^r},
\end{equation}
for all $t>0$  for which the Friedrich's approximation is valid. 
Note that the preceding inequality is a result of the use of a commutator estimate (see \cite[Section 2.10]{BahouriCheminDanchin2011}, for instance), which particularly  allows us to avoid having any extra derivatives on $q$ in the right-hand side.  
Since $H^r(\R)$ is an algebra and embeds into $L^\infty(\R)$, we further deduce that 
\begin{equation*}
	\frac{\dd\ }{\dd t} \|q(t,\cdot)\|_{H^r} \leqslant C \|q(t,\cdot)\|_{H^r}^2.
\end{equation*}
By a standard continuation argument, choosing specifically $T \leqslant (2C\|q_0\|_{H^r})^{-1}$, gives the uniform claimed estimate
\begin{equation*}
	\|q(t,\cdot)\|_{H^r} \leqslant 2 \|q_0\|_{H^r},
\end{equation*}
for all $t\leqslant T$.
 
 As soon as this a priori bound is obtained for the Friedrich's sequence, one can deduce in a routine way the existence, uniqueness and continuity of the local solution. 
 Indeed, this is a consequence of the fact that Friedrich's sequence is Cauchy in one lower norm, yielding compactness and, by interpolation, it follows that the solutions are in $C([0,T];H^r)$. 
 Moreover, the same transport estimate derived above applied to
the difference of two solutions gives uniqueness and continuous dependence.
All in all, since $\|u\|_{\dot H^s} \simeq \|q\|_{\dot H^{s-1}}$, this proves \eqref{propagation}.  

To conclude, notice that the continuation criterion follows from \eqref{q-estimate} and Gronwall's lemma.  
Agreement with the dissipative solution on the strong lifespan follows from uniqueness of the dissipative characteristics, previously proved in \cite{Dafermos2011}. 
This concludes the proof of the proposition.
\end{proof}

We prove now that the maximal time of existence of the strong solution coincides with the one at which the Lipschitz and the critical $\dot H^{\nicefrac{3}{2}}$ norms of the solution breakdown.

\begin{lem}[Critical slope inflation at wave breaking]\label{lem:critical-breaking}
Let $u_0\in C_c^\infty(\R)$, and denote $u$ the corresponding strong solution of
\eqref{HS} on the maximal time  $[0,T_*)$, with $T_*<\infty$. Then
\begin{equation*}
  \limsup_{t\uparrow T_*}\|\partial_x u(t,\cdot)\|_{\dot H^{1/2}}=\infty.
\end{equation*}
\end{lem}

\begin{proof}
Again, we use the notation $q=\partial_x u$, and we recall that the $L^2$ norm of $q$ is conserved before breaking.  
Suppose, for the sake of a contradiction, that 
\begin{equation*}
\|q(t,\cdot)\|_{H^{1/2}} \leqslant C_0 \quad   \text{on } [0,T_*).	
\end{equation*}
Then, for any fixed
$\sigma>1/2$, the standard transport estimate gives
\begin{equation}\label{high-q}
  \frac{\dd\ }{\dd t}  \|q(t,\cdot)\|_{H^\sigma}  \leqslant C \|q(t,\cdot)\|_{L^\infty}  \|q(t,\cdot)\|_{H^\sigma},
\end{equation}
before $T_*$. Owing to the one-dimensional logarithmic Sobolev inequality (see Proposition 2.104 from \cite{BahouriCheminDanchin2011})
\begin{equation*}
 	\norm {  q(t,\cdot)}_{L^\infty(\R)} \lesssim \norm { q(t,\cdot)}_{ H^{\frac{1}{2}}(\mathbb R)}   \log \left( e+  \frac{\norm { q(t,\cdot)}_{ H^{\sigma}(\mathbb R)}}{\norm {q(t,\cdot)}_{ H^{\frac{1}{2}}(\mathbb R)}}\right) ,
 \end{equation*}
 we infer that
\begin{equation*} 
  \|q(t,\cdot)\|_{L^\infty}  \leqslant C_{\sigma,C_0}  \Big(1+\log\bigl(e+\|q(t,\cdot)\|_{H^\sigma}\bigr)\Big).
\end{equation*}
Inserting this bound into \eqref{high-q}, and applying Osgood's lemma shows
that $\|q(t)\|_{H^\sigma}$ remains finite on $[0,T_*]$.  
This contradicts the continuation criterion in Proposition \ref{prop:propagation}.  
Thus, the inhomogeneous $H^{1/2}$ norm is unbounded. 
Since the $L^2$ norm is bounded, we conclude that its homogeneous $\dot H^{1/2}$ part is
unbounded as well, thereby completing the proof.
\end{proof}

 Next, we show that the solution previously constructed blows-up in finite time, for some well-chosen initial data (more precisely, we show that the solution does not blow-up if and only if the initial data is non-decreasing).  We also give some precise information about the blow-up solution, which will serve later on in the non-existence analysis.
 
 A crucial point here is that the blow-up time is explicitly given in terms of the Lipschitz norm of the initial data.

 \begin{pro}[Explicit blow-up time]\label{prop:blowup time}
Let $u_0$ be a smooth initial datum belonging to 
$  L^\infty(\R)\cap\dot H^1(\R)$, and let $u$ be the corresponding
strong solution on its maximal time interval $[0,T_*)$.  Then
\begin{equation}\label{Tstar}
 	T_*= 
 	\begin{cases}
 	\infty, &  \displaystyle \inf_{x\in \mathbb R} u_0'(x)	\geqslant 0, \\
 	-  \displaystyle\frac{2}{\inf_{x\in \mathbb R} u_0'(x)}, &  \displaystyle\inf_{x\in \mathbb R} u_0'(x) < 0.
 	\end{cases}
\end{equation}
\end{pro}

\begin{proof} 	
	We recall    that  $q= \partial_x u$  is governed by the equation 
\begin{equation*}
	\partial_t q + u \partial_x q = - \tfrac12 q^2.
\end{equation*}
Letting  $X_t$ denote the flow map associated to $u$, we set  
\begin{equation*}
	h(t,\xi)\eqdef q(t,X_t(\xi)).
\end{equation*}
Then, $h$ solves
\begin{equation*}
	\partial_t h(t,\xi) = -\tfrac12 h^2(t,\xi),  \qquad h(0,\xi)=u_0'(\xi),
\end{equation*}
which can be solved explicitly to find that 
\begin{equation}\label{h_formula}
	h(t,\xi)= \frac{2 h(0,\xi)}{2+t h(0,\xi)} ,
\end{equation}
as soon as the denominator is not identically zero.

 In the case when $\inf_x u_0'(x)	\geqslant 0$, it is readily seen that \eqref{h_formula} leads to a bound for $q$ in $ L^\infty ([0,T_*) \times \R)$, hence, we deduce that $T_*=\infty$. 

Let us now examine the case when there exists $x_0 \in \R$ such that $h(0,x_0)=u_0'(x_0)<0$.  Then, it is readily seen, from \eqref{h_formula},  that  $T_* \leqslant -2/u_0'(x_0)$, whence   $T_* \leqslant \inf_{x} -2/u_0'(x)$. Let now $t<\inf_{x} -2/u_0'(x)$. Therefore, we deduce from \eqref{h_formula}, again, that $q \in L^\infty ([0,t] \times \R)$ and, eventually from Lemma \ref{lem:critical-breaking}, that $T_* >t$. All in all, we have shown that $T_*=\inf_{x} -2/u_0'(x)$. This concludes the proof of the proposition.
\end{proof} 

\section{The subcritical regime}\label{section:subcritical}

This section is devoted to the proof of the strong ill-posedness of \eqref{HS} in the subcritical spaces. For convenience, we state below the precise  statement of result to prove here, which corresponds to a slightly stronger form of Theorem \ref{thm:main} in the case Sobolev space $\dot H^{s}(\mathbb R)$ with $s\in (1,\nicefrac{3}{2})$.

\begin{thm}\label{thm:supercritical}
	For every $s\in (1, \nicefrac{3}{2})$, there exists $u_0\in C^\infty \cap L^\infty(\mathbb R) \cap  \dot H^1(\mathbb R)\cap  \dot H^{s}(\mathbb R)$ such that the corresponding unique global dissipative solution $u$ of \eqref{HS} satisfies 
	\begin{equation*}
		 \underset{{t\in (0,T)}}{ \textnormal{ess sup}} \norm {u(t,\cdot)}_{\dot H^{ s}(\mathbb R)}  = \infty , \quad \text{for all } T>0.
	\end{equation*}
\end{thm}

 For a clean presentation of the proof of Theorem \ref{thm:supercritical}, we first establish a separate key elementary lemma about the time-evolution of a plateau before breaking time.  We recall that the slope of the solution $u$ is always denoted by $q$. Moreover, the active set $I_t$, the dissipative characteristics $X(t,\cdot)$ and the representation of the unique dissipative solution of \eqref{HS} are respectively given by \eqref{It}, \eqref{charX} and \eqref{charU}.

\begin{lem}[Time evolution of a plateau before breaking time]\label{lem:active:intervals}
	Assume that the initial data $u_0$ is given such that $q_0 =-A$ on some space interval $R$, for some $A>0$, where 
	\begin{equation*}
		A= -\inf_{x\in \mathbb R} q(x).
	\end{equation*}
	Then, for every $t\in (0,\nicefrac {2}{A})$, it holds that $R\subset I_t$, and the dissipative characteristic map $X(t,\cdot)$ and the solution along these characteristics  satisfy  
	\begin{equation*}
		\partial_\xi X (t,\xi) = \left( 1- \frac{tA}{2}\right)^2 \qquad \text{ and } \qquad q(t,X(t,\xi))= - \frac{A}{1-\tfrac{At}{2}},
	\end{equation*} 
	for a.e $(t,\xi) \in (0,\nicefrac {2}{A}) \times  R$.
	Moreover, it holds that  
	\begin{equation*}
		|X(t,R)|=\left( 1- \frac{tA}{2}\right)^2 |R|.
	\end{equation*}
\end{lem}

\begin{proof}
	The main point to use here is that, before the breaking time $t= \nicefrac {2}{A} $, the identities  \eqref{charX} and \eqref{charU}  are equivalently given by 
\begin{equation*} 
	X(t,\xi)  \bydef \xi+t u_0(\xi)   +\frac {t^2}{4} \int_{ -\infty }^{\xi}\big(u_0'(\zeta)\big)^2 \dd \zeta
\end{equation*} 
and 
\begin{equation*} 
	u(t,X(t,\xi))  =u_0(\xi)   +\frac t2      \int_{ -\infty}^{\xi}\big(u_0'(\zeta)\big)^2\dd\zeta.
\end{equation*} 
Thus, the first identities in the statement of the lemma follow by a direct differentiating in the $\xi$ variable, for a.e $\xi \in R$. As for the length of the interval $X(t,R)$, it is obtained from integrating $\partial_\xi X(t,\cdot)$ over $R$.
\end{proof}

We are now in a position to prove Theorem \ref{thm:supercritical} and, consequently, Theorem  \ref{thm:main} in the case $s\in (1,\nicefrac{3}{2})$.

\begin{proof}[Proof of Theorem \ref{thm:supercritical} and Theorem  \ref{thm:main} for $s\in (1,\nicefrac{3}{2})$]
	We split the proof into two steps for simplicity.

\subsubsection*{The initial data.}

 We introduce   
 \begin{equation*}
	  p\bydef \frac{2}{3-2s} \in (2,\infty),
\end{equation*}
which corresponds to the Lebesgue exponent from the Sobolev embedding 
\begin{equation*}
	\dot H^{s-1}(\mathbb R) \hookrightarrow L^p(\mathbb R).
\end{equation*}
 Next, let $Q\in C^{\infty}_c (\mathbb R)$, and $R\subset  \mathbb R$  be a non-trivial interval ($|R|>0$)  that are chosen in such a way that 
 \begin{equation*}
 	Q \geqslant -1, \quad Q|_R = -1 \quad \text{ and } \int_{\mathbb R} Q(x) \dd x = 0.
 \end{equation*}
 Accordingly, we define 
 \begin{equation*}
 	W(x)\bydef \int_{-\infty}^x Q(y) \dd y, \quad \text{ for all } x\in \mathbb R,
 \end{equation*}
 which is a $C^\infty_c (\mathbb R)$ function due to the vanishing mean property of $Q$. 
 
 Now, let $\varepsilon>0$ be a small parameter, and fix $(\rho  _n)_{n \geqslant 1}$ to be a sequence of positive small summable  numbers in the sense    that 
 \begin{equation*}
 	\sum_{n\geq 1}  \rho_{n} \leqslant \varepsilon ,
 \end{equation*} 
 Further take $(A_n)_{n\geqslant 1}$ to be a sequence tending to infinity (take $A_n= 2^{2n}$, for instance), and set 
 \begin{equation}\label{ell:def}
 	\ell_{n}\bydef \left(\frac{\rho_n}{A_n} \right)^{p}, \quad \text{for all } n \geqslant 1.
 \end{equation}
 Note that we can choose these sequences so that $\ell_n \in (0,1)$, for all $n \geqslant 1$.

Finally, we define the initial bubble and data by 
\begin{equation*}
	u_{0,n}(\cdot )\bydef A_{n} \ell _{n} W \left( \frac{\cdot- y_{n}}{\ell_{n}}\right) , \qquad u_0 \bydef \sum_{n\geqslant 1} u_{0,n},
\end{equation*}
where the translations $y_{n}\to \infty$ are chosen so that the supports of the bubbles above are pairwise disjoint and the family of supports is locally finite. Note that the corresponding slopes  are given by 
\begin{equation*}
	q_{0,n}(\cdot) = A_{n}Q \left( \frac{\cdot- y_{n}}{\ell_{n}}\right),
\end{equation*}
and, in particular, it holds that 
\begin{equation*}
	q_0=q_{0,n}=- A_{n}  \quad \text{ on } \quad  R_{n} \bydef y_{n} +\ell_{n} R,
\end{equation*}
where the first equality holds by virtue of the disjoint support of the initial bubbles. 

With the above setup, it is readily seen that 
\begin{equation*}
	\norm {u_{0,n}}_{L^\infty(\mathbb R)} + \norm {u_{0,n}}_{\dot H^1(\mathbb R)} \leqslant C_Q \rho _{n},
\end{equation*}
and
\begin{equation*}
	\norm {u_{0,n}}_{\dot H^s(\mathbb R)} =  \rho _{n} \norm W_{\dot H^s(\mathbb R)},
\end{equation*}
where we used that $\ell_{n}\in (0,1)$ and $p\in (2,\infty)$ with the definition of $\ell_n$ in \eqref{ell:def}. To conclude this step, we notice that the local finiteness of the supports gives that $u_0 \in C^\infty(\mathbb R)$, while the triangular inequality ensures that  $u_0 \in L^\infty \cap \dot H^s (\mathbb R)$, for all $s\in [1,\nicefrac{3}{2})$, with a norm that can be made arbitrary small.

\subsubsection*{Loss of regularity at the prescribed times.}

We now choose another sequence of small positive numbers $(\delta_n)_{n\geqslant 1} \subset (0,\nicefrac{1}{4})$, tending to zero, given by the relation 
\begin{equation}\label{delta:def}
	 \rho_n \delta_n ^{-1 + \frac{2}{p}} = n.
\end{equation}
We denote by $T_n$ the breaking time of the solution corresponding to the $n^{th}$ bubble, i.e.,
\begin{equation*}
	T_n\bydef \frac{2}{A_n}.
\end{equation*}
Moreover, defining the observation time interval by 
\begin{equation*}
	 \mathcal  I_n\bydef [T_n(1-2\delta_n) , T_n(1-\delta_n)] \subset (0,T_n),
\end{equation*}
we obtain, by applying Lemma \ref{lem:active:intervals} for any a.e $t\in \mathcal I_n$ and $x\in  X(t, R_{n})$,  that 
\begin{equation*}
	|q(t,x)| = \frac{A_{n}}{1-\frac{A_n t}{2}} , \qquad |X(t, R_{n})|= \left( 1-\frac{A_n t}{2}\right)^2 \ell_{n} |R|.
\end{equation*}
Notice in passing that, as soon as $t\in \mathcal I_n$, it holds that
\begin{equation}\label{delta:bound}
	\delta_n \leqslant 1-\frac{A_n t}{2} \leqslant 2\delta_n.
\end{equation}
Therefore, we find by combining \eqref{ell:def}, \eqref{delta:def} and \eqref{delta:bound} that 
\begin{equation*}
	\begin{aligned}
		\norm {q(t, \cdot)}_{L^p(\mathbb R)}^p \geqslant \int_{X(t, R_{n})} |q(t, x)|^p \dd x
		&= \left(   \frac{A_{n}}{1-\frac{A_n t}{2}}\right)^p|X(t,R_n)|
		\\
		&= |R| \ell_n A_n^p \left(  1-\frac{A_n t}{2}\right)^{2-p}
		\\
		&= |R| \rho_n^p \left(  1-\frac{A_n t}{2}\right)^{2-p}
		\\
		&\geqslant 2^{2-p} |R| \rho_n^p \delta_n ^{2-p}
		\\
		&= 2^{2-p} |R|n^p,
		\end{aligned}
\end{equation*}
for a.e $t\in \mathcal I_n$. 
By Sobolev embedding $\dot H^{s-1}\hookrightarrow L^p(\mathbb R)$, we deduce that 
\begin{equation*}
	\norm {u(t,\cdot)}_{\dot H^s(\mathbb R)} \gtrsim n, \quad \text{ for a.e } t\in \mathcal I_n.
\end{equation*}
Finally, towards a contradiction, suppose that there exists $T>0$ such that 
\begin{equation*}
	u\in L^\infty([0,T]; \dot H^s(\mathbb R)).
\end{equation*}
As 
\begin{equation*}
	\sup \{t\in \mathcal I_n \} \leqslant T_n \to 0 \quad \text{ as } n\to\infty,
\end{equation*}
we deduce that, for $T$ being already fixed, there exist $N=N(T)$ for which $\mathcal I_n \subset (0,T)$ for all $n\geqslant N$. Thus, as $\mathcal I_n$ has a positive measure, we deduce that 
\begin{equation*}
	\underset{{t\in (0,T)}}{ \textnormal{ess sup}} \norm {u(t,\cdot)}_{\dot H^s(\mathbb R)} \gtrsim n, \quad \text{ for all } n\geqslant N.
\end{equation*}
By taking $n$ to infinity, this contradicts the boundedness of $u$ in $L^\infty([0,T]; \dot H^s(\mathbb R))$, thereby completing the proof. 
\end{proof}

\section{The critical regime}\label{section:critical}

The approach to the strong ill-posedness in the critical regime corresponding to $s=\nicefrac{3}{2}$ is slightly more involving. We thus present the necessary building blocks for the ultimate proof into two separate paragraphs. The proof of the non-persistence of the critical Sobolev regularity will then be discussed at the end of this section.

\subsection{Critical profiles and separated seminorms}\label{sec:profiles}

In this section, we present the building blocks for the initial data which will be used in the proof of Theorem \ref{thm:main}. 
The construction below isolates the endpoint defect $\dot H^{3/2} \not\hookrightarrow \dot W^{1,\infty} (\mathbb R)$ while retaining compact support and small energy.

\begin{lem}[A frequency-localized profile]\label{lem:frequency-profile}
Let $b \in (1/2,1)$ be fixed. 
Then, there exist two constants $c>0$ and $C>0$ such that, for every $\delta \in (0,1)$ and every sufficiently large integer $K$, there exists a real-valued smooth function $\Phi_{\delta,K}$ satisfying
\begin{align}\label{Phi-small}
	  \|\Phi_{\delta,K}\|_{\dot H^{3/2}} + \|\Phi_{\delta,K}\|_{\dot H^1} + \|\Phi_{\delta,K}\|_{L^\infty} & 	  \leqslant C \delta,  \\ \label{Phi-large}
  -\Phi_{\delta,K}'(0) &  \geqslant c \delta K^{1-b}.
\end{align}
Moreover, for every $\sigma>0$, there exists $C_\sigma>0$, which is independent of $\delta$, such that  
\begin{equation}\label{Phi-Besov}
  \|\Phi_{\delta,K}\|_{\dot B^{3/2-\sigma}_{2,1}}
  \leqslant C_\sigma \delta.
\end{equation}
\end{lem}

\begin{proof}
Let $\chi_K\in C_c^\infty(\mathbb R)$ be a smooth even cut-off function satisfying
\begin{equation*}
	0 \leqslant \chi_K(\xi) \leqslant 1,  \qquad  \chi_K(\xi) = 1 \quad\text{for } 2 \leqslant |\xi| \leqslant 2^K,
\end{equation*}
and
\begin{equation*}
\mathrm{supp} \chi_K  \subset \left\{\xi\in \mathbb R : \tfrac32 \leqslant |\xi| \leqslant 2^{K+1} \right\}.
\end{equation*}
We then define
\begin{equation}\label{Phi-hat}
\widehat \Phi_{\delta,K}(\xi) \bydef i\delta \ \mathrm{sign}(\xi) \frac{\chi_K(\xi)}{|\xi|^2(\log_2|\xi|)^b},
\end{equation}
for all $\xi\in \mathbb R$. 
As the Fourier transform of  $\Phi_{\delta,K}(\xi)$ is imaginary and odd, by construction, it then follows that 
$\widehat\Phi_{\delta,K}(-\xi)=\overline{\widehat\Phi_{\delta,K}(\xi)}$ and, thus, $\Phi_{\delta,K}(\xi)$ is real.  Moreover, a direct
calculation gives
\begin{equation*}
	  \|\Phi_{\delta,K}\|_{\dot H^{3/2}}^2  \lesssim \delta^2 \int_{\tfrac32}^{2^{K+1}} \frac{\dd r}{r(\log_2r)^{2b}}   \lesssim \delta^2,
\end{equation*}
due to the assumption that $2b>1$.  Similarly, it is readily seen that 
\begin{equation*}
	  \|\Phi_{\delta,K}\|_{\dot H^1}^2  \lesssim \delta^2 \int_{\tfrac32}^{2^{K+1}} \frac{\dd r}{r^2}  \lesssim \delta^2,  \qquad		  \|\Phi_{\delta,K}\|_{L^\infty}  \lesssim \|\widehat\Phi_{\delta,K} \|_{L^1} \lesssim \delta.
\end{equation*}
On the other hand, with the sign in \eqref{Phi-hat}, we observe that one has at the origin  
\begin{equation*}
	  -\Phi_{\delta,K}'(0)  = c \delta \int_{\tfrac32}^{2^{K+1}} \frac{\chi_K(r) \dd r}{r(\log_2r)^b}  \geqslant c \delta K^{1-b}.
\end{equation*}
Finally, at the dyadic scale $2^j$, one computes that 
\begin{equation*}
	  2^{3j/2} \| \dot\Delta_j \Phi_{\delta,K} \|_{L^2} \lesssim  \delta j^{-b}.
\end{equation*}
Therefore, multiplying by $2^{-j\sigma}$ and summing up, we obtain  \eqref{Phi-Besov}. This completes the proof of the lemma.
\end{proof}

By building on the profiles constructed in the preceding lemma, we now finalize the construction of the building blocks for the initial data. The key additional property of the profiles constructed next, which we call bubbles, is their compact spatial support.

\begin{lem}[Compact critical bubbles]\label{lem:bubbles}
There exist two constants $c>0$ and $C>0$ and a family $(\varphi_\varepsilon)_{0<\varepsilon<1}$ of functions in $C_c^\infty(\R)$ with the properties
\begin{gather}\label{bubble-support}
  \supp\varphi_\varepsilon \subset[-\varepsilon,\varepsilon], \\ \label{bubble-small}
  \|\varphi_\varepsilon\|_{\dot H^{3/2}}  +\|\varphi_\varepsilon\|_{\dot H^1}  +\|\varphi_\varepsilon\|_{L^\infty}   \leqslant C \varepsilon^2, \\ \label{bubble-large}
  -\inf_{x\in\R}\varphi_\varepsilon'(x)   \geqslant -\varphi_\varepsilon'(0)\geqslant c \varepsilon^{-1},
\end{gather}
for all $\varepsilon\in (0,1)$.
\end{lem}

\begin{proof}
Let $\psi\in C_c^\infty(\R)$ be such that  $0\leqslant \psi \leqslant 1$, $\psi=1$ on $[-1/2,1/2]$, and $\supp\psi\subset[-1,1]$.  
Given the notations from the proof of Lemma \ref{lem:frequency-profile}, we further introduce a large parameter $N$  by setting $N=K^{1-b}$, and  
\begin{equation*}
	 \delta=N^{-7/8}, \qquad \varepsilon = N^{-1/8}, \qquad   \varphi_\varepsilon(x) \bydef \psi(x/\varepsilon) \Phi_{\delta,K}(x).
\end{equation*}
Accordingly, we have that  
\begin{equation*}
	\supp \varphi_\varepsilon  \subset[-\varepsilon ,\varepsilon] \qquad  \text{and} \qquad 
-\varphi_\varepsilon'(0)  =  -\Phi_{\delta,K}'(0)  \geqslant  c N^{1/8} = c\varepsilon^{-1},
\end{equation*}
which establishes assertions \eqref{bubble-support} and \eqref{bubble-large} altogether.

It now remains to estimate the norms after the physical cut-off. 
To that end, we employ standard paraproduct estimates (see \cite[Theorems 2.47 and 2.52]{BahouriCheminDanchin2011}, for instance) to obtain, for
any fixed $\gamma>0$, that
\begin{align*}
 \|\psi(\cdot/\varepsilon) \Phi_{\delta,K}\|_{\dot H^{3/2}}  & \lesssim \|\psi(\cdot/\varepsilon)\|_{L^\infty}   \|\Phi_{\delta,K}\|_{\dot H^{3/2}} + \|\psi(\cdot/\varepsilon)\|_{\dot H^{3/2+\gamma}}  \|\Phi_{\delta,K}\|_{\dot B^{-\gamma}_{\infty,\infty}}  \\
 &\lesssim \delta + \varepsilon^{-(1+\gamma)}   \|\Phi_{\delta,K}\|_{\dot H^{1/2-\gamma}},
\end{align*}
where the first term on the right hand-side comes from the low-high frequencies  and the reminder, while the second one is the contribution from the high-low frequencies of the estimated product.
The last norm is bounded by $C_\gamma \delta$ by the same dyadic computation
as in Lemma \ref{lem:frequency-profile}. 
Taking $\gamma=4$ gives
\begin{equation*}
	  \|\varphi_\varepsilon\|_{\dot H^{3/2}}  \lesssim N^{-7/8} + N^{5/8} N^{-7/8}  \lesssim N^{-1/4} = \varepsilon^2.
\end{equation*}
The same argument at order one, together with the $L^\infty$ estimate in
\eqref{Phi-small}, yields
\begin{equation*}
	  \|\varphi_\varepsilon\|_{\dot H^1}  + \|\varphi_\varepsilon\|_{L^\infty} \lesssim \varepsilon ^2,
\end{equation*}
thereby concluding the proof of the lemma.
\end{proof}

We next record the almost-orthogonality statement used for the differentiated increments.

\begin{lem}[Separated $\dot H^{1/2}$ seminorms]\label{lem:separated}
Let $(f_k)_{k\in \mathbb N}$ be a sequence of measurable functions in $L^1(\R)$ such that
\begin{equation*}
	  \sup_{k\in \mathbb N} \|f_k\|_{L^1 (\mathbb R)} \leqslant A,  \qquad  \supp f_k \subset[X_k-R,X_k+R]
\end{equation*}
for a fixed $A,R>0$, and some sequence of positions $(X_k)_{k \in \mathbb N}$.  
Assume further that 
\begin{equation}\label{exponential-separation}
  X_{k+1} - X_k \geqslant 4^k + 2 R, 
\end{equation}
for all $k\in \mathbb N$.
Then, there is a constant $C=C(A,R)$ such that, for every $N$,
\begin{equation}\label{almost-orthogonality}
  \left|  \left\| \sum_{k=1}^N f_k\right\|_{\dot H^{1/2}}^2 -\sum_{k=1}^N \|f_k\|_{\dot H^{1/2}}^2  \right| \leqslant C.
\end{equation}
\end{lem}

\begin{proof}
For the finite sum
\begin{equation*}
	f^N = \sum_{k=1}^N f_k,
\end{equation*}
 writing   
\begin{equation*}
	\left |\sum_{k=1}^N \left(  f_k(x) -  f_k(y)\right)\right |^2  = \sum_{k=1}^N |f_k (x)- f_k(y)|^2 + 2 \sum_{1\leqslant i < k \leqslant N} ( f_i(x)- f_i (y))( f_k(x)- f_k (y)),
\end{equation*}
 for all $x,y \in \mathbb R$, it then follows by the disjointness of the supports that 
\begin{align}\label{increments0}
 |f^N(x) - f^N(y)|^2  & = \sum_{k=1}^N|f_k(x)-f_k(y)|^2  -2 \sum_{\substack{i,k=1\\i\ne k}}^N f_i(x)f_k(y).
\end{align}
For $i\neq k$, \eqref{exponential-separation} implies 
\begin{equation*}
	  \dist(\supp f_i,\supp f_k) \geqslant c 2^{i+k}
\end{equation*}
for some universal constant $c>0$. Therefore, we deduce that 
\begin{equation*}
	  \sum_{i\ne k}\iint_{\R^2}  \frac{|f_i(x)f_k(y)|}{|x-y|^2}\dd x\dd y  \leqslant C A^2 \sum_{i\ne k}2^{-i-k}<\infty,
\end{equation*}
uniformly in $N$.  
Thus, integrating the  identity \eqref{increments0} against $|x-y|^{-2}\dd x\dd y$ proves \eqref{almost-orthogonality}.  
The assertion for infinite sums follow by applying the estimate to partial sums and using Fatou's lemma.
\end{proof}

\subsection{One-sided localization}\label{sec:localization}

The nonlocal forcing in \eqref{HS} prevents a usual finite-speed statement.
What replaces it is a one-sided domain of dependence encoded by \eqref{charX}--\eqref{charU}.

Once again, we will use the  explicit representation of the unique global dissipative solution given by \eqref{charU}.

\begin{pro}[Evolution of one-sided supports]\label{prop:support} 
Let $u_0 \in C^\infty(\R)\cap L^\infty(\mathbb R)\cap\dot H^1(\mathbb R)$, and let $u$ be the associated global
dissipative solution.
\begin{enumerate}[label=\textup{(\roman*)}]
\item If $\supp u_0\subset[a,\infty)$, then
\begin{equation*}
	 \supp u(t,\cdot)\subset[a,\infty),  \qquad \text{for every } t \geqslant 0.
\end{equation*}
\item If $\supp u_0\subset(-\infty,b]$, then
\begin{equation}\label{right-q-support}
  \supp \partial_x u(t,\cdot)  \subset \left(-\infty,b+\frac{t^2}{4}\|u_0'\|_{L^2}^2\right],  \qquad\text{for every } t \geqslant 0.
\end{equation}
Moreover, $u(t,\cdot)$ is constant on the right of the endpoint in
\eqref{right-q-support}.
\end{enumerate}
\end{pro}

\begin{proof}
If $\xi < a$, then $u_0(\xi)=0$ and $u_0'=0$ on $(-\infty,\xi)$.  
In this case, we notice that 
\begin{equation*}
		0 \leqslant \int_0^t \int_{I_\tau \cap (-\infty,\xi )} ( u_0'(\zeta ))^2  \dd \zeta \dd \tau
		 \leqslant \int_0^t \int_{ -\infty}^{\xi }  ( u_0'(\zeta   ))^2   \dd \zeta \dd \tau
		 =0,
	\end{equation*}
whence, formulas \eqref{charX} and \eqref{charU} give that $X(t,\xi)=\xi$ and $u(t,\xi)=0$, thereby proving (i).

We turn now to the proof of (ii). If $\xi>b$, then $u_0(\xi)=0$ and the inner integrals in
\eqref{charX}--\eqref{charU} no longer depend on $\xi$.  Thus, with
\begin{equation*}
	  \kappa(\tau)\bydef \frac12 \int_{I_\tau}( u_0'(\zeta   ))^2\dd\zeta = \frac12 \int_{I_\tau \cap (-\infty,\xi )}   ( u_0'(\zeta   ))^2  \dd \zeta \dd \tau  \leqslant \frac12 \|u_0'\|_{L^2(\mathbb R)}^2,
\end{equation*}
it follows that 
\begin{equation*}
	  X(t,\xi) = \xi+\int_0^t(t-\tau) \kappa(\tau) \dd \tau, \qquad  u(t,X(t,\xi)) = \int_0^t \kappa(\tau) \dd \tau.
\end{equation*}
Therefore, it is readily seen that 
\begin{equation*}
		u\left  (t,x \right)=     \int_0^t  \kappa(\tau ) \dd \tau,
\end{equation*}
for any $x>b+\int_0^t (t-\tau) \kappa(\tau) \dd \tau $ and $t\geqslant 0$.
The displacement is at most $t^2\|u_0'\|_{L^2(\mathbb R)}^2/4$, whence, beyond the endpoint in \eqref{right-q-support}, $u$ is spatially constant, and $\partial_x u=0$. This completes the proof of the proposition.
\end{proof}

The next proposition establishes several essential properties of two solutions arising from initial data with a common one-sided domain of dependence. 
These properties will play a central role in the proof of the nonexistence of solutions to \eqref{HS} in critical Sobolev spaces where the essence of our argument relies on estimating a superposition of dissipative solutions generated by carefully chosen bubbles of initial data. 
In particular, the support properties established below will constitute a cornerstone of the nonexistence proof.

\begin{pro}[One-sided domain of dependence]\label{prop:domain}
Let $u_{0,1},u_{0,2} \in L^\infty(\mathbb R)\cap\dot H^1(\mathbb R)$, and denote by $u_1,u_2$ be their dissipative
solutions, respectively.  
Assume further that
\begin{equation*}
	  u_{0,1} = u_{0,2} \quad \text{on }(-\infty,a],
\end{equation*}
for some $a\in \mathbb R$. Then, setting $M=\max_i\|u_{0,i}\|_{L^\infty}$, it holds that
\begin{equation}\label{domain}
  u_1(t,x) = u_2(t,x)  \qquad \text{whenever } x \leqslant a-tM.
\end{equation}
If, in addition, $\supp u_{0,i}\subset(-\infty,b_i]$, for some $b_1,b_2\in \mathbb R$, then one has that
\begin{equation}\label{difference-right-support}
  \supp \bigl(\partial_xu_2(t,\cdot)-\partial_xu_1(t,\cdot) \bigr)  \subset \left(-\infty,   \max_i\left(b_i + \frac{t^2}{4}\|u_{0,i}'\|_2^2\right) \right],
\end{equation}
for all $t\geqslant 0$.
\end{pro}

\begin{proof}
For every $\xi \leqslant a$, the initial values, their almost-everywhere derivatives,
and the truncated energy integrals over $(-\infty,\xi)$ agree.  Therefore
\eqref{charX}--\eqref{charU} yields that
\begin{equation*}
	  X_1(t,\xi) = X_2(t,\xi),  \qquad  u_1(t,X_1(t,\xi)) = u_2(t,X_2(t,\xi)).
\end{equation*}
Moreover, the characteristic maps are nondecreasing (see the proof of Theorem 1 in \cite{BressanConstantin2005}) and satisfy
\begin{equation*}
	  X_i(t,a)\geqslant a + t u_{0,i}(a) \geqslant a-tM,
\end{equation*}
because the double integral in \eqref{charX} is nonnegative.  
Their images of $(-\infty,a]$ therefore contain $(-\infty,a-tM]$, which proves
\eqref{domain}.  
The additional support statement follows by applying Proposition \ref{prop:support} to each solution separately. 
This completes the proof of the proposition.
\end{proof}

\subsection{Strong ill-posedness in the critical case}\label{sec:proof-main}

We are now in a position to prove the ill-posedness of \eqref{HS} in critical Sobolev space $\dot H^{\nicefrac{3}{2}}(\mathbb R)$.  

\begin{proof}[Proof of Theorem \ref{thm:main} for $s=\nicefrac{3}{2}$]
Let $(\varphi_\varepsilon)_{0<\varepsilon<1}$ be the family from Lemma \ref{lem:bubbles}, and choose a sequence $\varepsilon_n \downarrow 0$ such that
\begin{equation}\label{epsilon-choice}
  \varepsilon_n^2 \leqslant 2^{-n}, \qquad \text{with }  a_n \bydef -\inf_{x\in\R}\varphi_{\varepsilon_n}'(x)
\end{equation}
is strictly increasing, and $a_n \geqslant n$.  
This is possible by virtue of Lemma \ref{lem:bubbles}.  
To lighten the notations, we further set $\varphi_n=\varphi_{\varepsilon_n}$.

Next, we choose the centers $(\alpha_n)_{n\in \mathbb N}$ as follows.  Let
\begin{equation}\label{kappa}
  E_* \bydef \sum_{n=1}^\infty\|\varphi_n'\|_{L^2}^2 < \infty,  \qquad
  M_* \bydef \sum_{n=1}^\infty\|\varphi_n\|_{L^\infty} < \infty, \qquad
  R \bydef 2 + M_* + \frac14 E_*,
\end{equation}
and we define
\begin{equation}\label{centers}
  \alpha_1 = 1,  \qquad   \alpha_{n+1} = \alpha_n + 4^n + 2R.
\end{equation}
Accordingly, we now introduce the initial bubbles by setting
\begin{equation}\label{initial-bubbles}
  u_{0,n}(x) \bydef \varphi_n(x-\alpha_n), \qquad
  v_{0,n} \bydef \sum_{k=1}^n u_{0,k},   \qquad
  v_0 \bydef \sum_{k=1}^\infty u_{0,k}.
\end{equation}
By \eqref{bubble-small} and \eqref{epsilon-choice}, the last series converges in $L^\infty(\R)\cap\dot H^1(\R)\cap\dot H^{3/2}(\R)$.  
Thus $v_0$ is an admissible datum for the global dissipative flow.

Let $v_n$ and $v$ be the dissipative solutions generated by $v_{0,n}$ and $v_0$, respectively, and we further define the increments
\begin{equation}\label{increments}
  z_1 \bydef v_1,  \qquad   z_n \bydef v_n-v_{n-1}  \quad  (n \geqslant 2).
\end{equation}
We emphasize that the $z_n$ are increments of solutions, not solutions with
initial data $u_{0,n}$.

Now, we claim that, for $0\leqslant t \leqslant 1$,
\begin{equation}\label{increment-support}
  \supp \partial_x z_n(t,\cdot)   \subset [\alpha_n-R, \alpha_n+R].
\end{equation}
Indeed, as $v_{0,n}$ and $v_{0,n-1}$ agree on $(-\infty,\alpha_n-1]$, and their $L^\infty$ norms are bounded by $M_*$, then Proposition \ref{prop:domain} shows that $z_n(t,x)=0$ for all $x \in  (-\infty, \alpha_n-1-tM_*] \supset (-\infty,  \alpha_n-R]$.   
On the other hand, as $v_{0,n}$ and $v_{0,n-1}$ are supported in left half-lines with right endpoints at most $\alpha_n+1$, and their energies are at most $E_*$, then the second part of Proposition \ref{prop:domain} gives that  $\partial_x z_n(t,x) = 0$ for all
$x \geqslant \alpha_n+1+t^2E_*/4 \leqslant \alpha_n+R$.  
This proves \eqref{increment-support}.  

Notice that the same one-sided comparison, now without a right-support assumption, gives
the pointwise representation
\begin{equation}\label{pointwise-series}
  v(t,x) = \sum_{n=1}^\infty z_n(t,x), \qquad (t,x) \in [0,1] \times \R.
\end{equation}
Indeed, to see this, let  $(t,x)$ be fixed, and choose $n_0$ so large that $x \leqslant \alpha_{n_0+1}-1-tM_*$.  
As the data $v_0$ and $v_{0,n}$ agree on $(-\infty,\alpha_{n_0+1}-1]$ for every $n \geqslant n_0$, then
Proposition \ref{prop:domain} yields $v(t,x)=v_n(t,x)$.  
Therefore, telescoping \eqref{increments} proves \eqref{pointwise-series}.

Now, observe that the supports in \eqref{increment-support} are exponentially separated by \eqref{centers}, and that we also have a uniform $L^1$ bound for the differentiated increments.  
Indeed, by H\"older's inequality, the dissipative energy inequality, and \eqref{increment-support}, yield that
\begin{align}\nonumber
 \|\partial_xz_n(t,\cdot)\|_{L^1}  & \leqslant (2R)^{1/2}\|\partial_xz_n(t,\cdot)\|_{L^2} \\ \label{L1-increments}
 & \leqslant (2R)^{1/2}  \bigl(\|\partial_xv_n(t,\cdot)\|_{L^2} + \|\partial_xv_{n-1}(t,\cdot)\|_{L^2}\bigr)
 \leqslant 2(2RE_*)^{1/2},
\end{align}
uniformly in $n$ and $t \in [0,1]$.

Because the supports of the initial slopes are disjoint and the numbers $a_n$ are strictly increasing, it happens that 
\begin{equation*}
	  \inf_x v_{0,n}'(x) = \inf_x u_{0,n}'(x) = -a_n.
\end{equation*}
Hence, by Proposition \ref{prop:blowup time}, the strong solution $v_n$ blows-up at time
\begin{equation}\label{Tn}
  T_n \bydef \frac{2}{a_n}.
\end{equation}
Notice that the sequence $(T_n)_{n \in \mathbb N}$ is strictly decreasing and tends to zero.  
Moreover, $v_{n-1}$ remains smooth up to and beyond $T_n$.  
Therefore, by Lemma \ref{lem:critical-breaking}, we have
\begin{equation*}
  \limsup_{t \uparrow T_n}  \|\partial_x v_n(t,\cdot)\|_{\dot H^{1/2}} = \infty,	
\end{equation*}
and hence, by \eqref{increments}, one obtains
\begin{equation}\label{increment-inflation}
  \limsup_{t \uparrow T_n}  \|\partial_x z_n(t,\cdot)\|_{\dot H^{1/2}} = \infty.
\end{equation}

To conclude the proof, we take now any $T>0$, and choose $n$ so large that $T_n < \min\{T,1\}$.
Additionally, we select times $t_j \uparrow T_n$ for which the norm in \eqref{increment-inflation} tends to infinity.  
At each $t_j$, we now apply Lemma \ref{lem:separated} to $f_k = \partial_x z_k(t_j)$, using \eqref{increment-support}, \eqref{centers}, and \eqref{L1-increments}.  
Together with \eqref{pointwise-series}, this yields that 
\begin{equation*}
  \|\partial_x v(t_j,\cdot)\|_{\dot H^{1/2}}^2  \geqslant \|\partial_xz_n(t_j,\cdot)\|_{\dot H^{1/2}}^2-C,	
\end{equation*}
where $C$ is independent of $j$.  
Therefore
\begin{equation*}
\sup_{0<t<T} \|v(t,\cdot)\|_{\dot H^{3/2}}  = \sup_{0<t<T} \|\partial_xv(t,\cdot)\|_{\dot H^{1/2}} = \infty,
\end{equation*}
which establishes \eqref{main} and completes the proof of the theorem.
\end{proof}

\begin{rem}[On the stronger all-time statement]
The argument proves arbitrarily large critical norm at arbitrarily small positive times and therefore failure of persistence on every interval.  
It does not, by itself, imply $v(t) \notin \dot H^{3/2}$ for every fixed $t>0$.
Such a statement would require an additional analysis of the Eulerian singularity created at the boundary of the collapsed-label set $\{ u_0' \leqslant -2/t \}$ in the dissipative formula.  
We do not make that stronger claim here.  
\end{rem}

\bibliographystyle{plain} 
\bibliography{bib_file}

\end{document}